\documentclass [11pt,a4paper]{amsart}
\usepackage[utf8]{inputenc}
\usepackage{amsfonts,amsmath,amsthm}
\usepackage{amssymb,latexsym}
\usepackage{mathtools} 

\usepackage{latexsym}
\usepackage{mathrsfs}
\usepackage[hypcap]{caption}

\usepackage{enumitem}

\newtheorem{theorem}{Theorem}[section]
\newtheorem{lemma}[theorem]{Lemma}
\newtheorem{corollary}[theorem]{Corollary}
\newtheorem{proposition}[theorem]{Proposition}
\newtheorem*{theorem*}{Theorem}

\theoremstyle{definition}

\newtheorem{definition}[theorem]{Definition}
\newtheorem{remark}[theorem]{Remark}

\numberwithin{equation}{section}

\usepackage{hyperref}
\hypersetup{colorlinks,citecolor=blue,linkcolor=blue}

\newcommand\vol{\mathrm{vol}}
\newcommand\area{\mathrm{area}}
\newcommand\tr{\mathrm{tr}}
\newcommand\sys{\mathrm{sys}}

\DeclareMathOperator{\N}{N}
\DeclareMathOperator{\GL}{GL}

\DeclareMathOperator{\im}{Im}

\DeclareMathOperator{\Spin}{Spin}

\DeclareMathOperator{\PSp}{PSp}

\DeclareMathOperator{\Isom}{Isom}

\newcommand{\V}{\mathcal{V}}
\newcommand{\Vf}{\V_{\mathrm{f}}}

\newcommand{\Vi}{\V_\infty}

\usepackage{amssymb}
\usepackage{amsthm}
\usepackage{amscd}

\usepackage{graphicx}%
\usepackage{fancyhdr}

\def\Q{\mathbb{Q}}

\def\H{\mathbb{H}}
\def\O{\mathcal{O}}
\def\o{\mathfrak{o}}
\def\p{\mathfrak{p}}

\def\RR{\mathbb{R}}

\newcommand\C{\mathbb{C}}

\newcommand{\HH}{\mathbf{H}_\H}

\DeclareMathOperator{\Sp}{Sp}
\DeclareMathOperator{\G}{\mathbf{G}}
\newcommand{\U}{\mathbf{U}}
\DeclareMathOperator{\R}{Re}

\usepackage{tcolorbox}

\begin{document}

\title[Systole of quaternionic hyperbolic manifolds]
        {On the systole growth in congruence quaternionic hyperbolic manifolds}
\author{Vincent Emery}
\author{Inkang Kim}
\author{Plinio G. P. Murillo}

\address{Vincent Emery: Universitat Bern, Mathematisches Institut, Sidlerstrasse 5, CH-3012 Bern, Switzerland}
\email{vincent.emery@math.ch}

\address{Inkang Kim: School of Mathematics, KIAS, Heogiro 85, Dongdaemun-gu, Seoul, 02455, Republic of Korea}
\email{inkang@kias.re.kr}

\address{Plinio Murillo: School of Mathematics, KIAS, Heogiro 85, Dongdaemun-gu, Seoul, 02455, Republic of Korea}
\email{plinioguillel@gmail.com}

\thanks{I. Kim gratefully acknowledges the partial support of Grant NRF-2019R1A2C1083865. 
V. Emery supported by SNSF project no PP00P2\_183716}


%\maketitle

\begin{abstract}
We provide an explicit lower bound for the systole in principal congruence covers
of compact quaternionic hyperbolic manifolds.  %, in terms of the dimension and the volume.
We also prove the optimality of this lower bound. 
%In this paper we prove that the systole and volume of principal congruence covers $M_{I}$ of a compact quaternionic hyperbolic %orbifold $M$, up to passing to a commensurable class, associated to prime ideals with norm large enough satisfy the relation
%\begin{equation*}
    %\sys_{1}(M_{I})\geq\frac{4}{(n+1)(2n+3)}\log\big(\vol(M_{I})\big)-d,
%\end{equation*}
%where $d$ is a constant independent of $I$. We also prove that the multiplicative constant $\frac{4}{(n+1)(2n+3)}$ is optimal.
\end{abstract}
%\footnotetext[1]{2000 {\sl{Mathematics Subject Classification.}}
%51M10, 57S25.} 
%\footnotetext[2]{{\sl{Key words and phrases.}}
%Congruence subgroups, Systole, Arithmetic subgroup, Quaternionic hyperbolic space.} 
%\footnotetext[3]{\sl{I. Kim gratefully acknowledges the partial support of   Grant NRF-2019R1A2C1083865.}}
\maketitle

%%%%%%%%%%%%%%%%%%%%%%%%%%%%%%%%%%%%%%%%%%%%%%%%%%%%%%%%%%%%%%%%%%%%%%%%%%%%%
\section{Introduction}

One the most important quantity associated with a Riemannian manifold $M$ is the shortest geodesic
length, which is called the \textit{systole} of $M$. We will denote it by $\sys_1(M)$.
Buser and Sarnak \cite{BS} constructed examples of hyperbolic surfaces $S$ whose systole grows
logarithmically with respect to the area: $$\sys_{1}(S)\geq \frac{4}{3}\log (\area(S))-c,$$ where
$c$ is a constant independent of $S$. Indeed these surfaces are congruence coverings of an
arithmetic hyperbolic surface. In 1996, Gromov \cite[Sec. 3.C.6]{Gr} showed that for any regular
congruence covering $M_I$ of a compact arithmetic locally symmetric space $M$, there exists a
constant $C>0$ so that $$ \sys_{1}(M_I)\geq C \log(\vol(M_I))- d,$$ where $d$ is independent of
$M_I$. This method, however, does not provide an explicit value for the constant $C$. The knowledge
of a precise value for this constant  gives us not only geometrical information on the locally
symmetric space, but it has also proven  useful in applications to other contexts, see for instance
the discussion in \cite[Prop. 5.3]{Bel13}, and \cite[Sec. IV]{GL14}.

The optimal value of the constant $C$ mainly depends on the Lie group type of the associated isometry
group, and there are several cases where an explicit $C$ are calculated. In 2007, Katz, Schaps and
Vishne \cite{KSV07} generalized Buser and Sarnak's result to any compact arithmetic hyperbolic
surface. They also proved that for compact arithmetic hyperbolic 3-manifolds the constant
$C=\frac{2}{3}$ works. For real arithmetic hyperbolic $n$-manifold of the first type, these results
were generalized by Murillo in \cite{Murillo}, where he proved that the constant is equal to
$\frac{8}{n(n+1)}$. In an appendix to this article, D\'oria and Murillo proved that this is the best
possible constant in this case. A similar result for Hilbert modular varieties was obtained in
\cite{Murillo17}.

Recently, Lapan, Linowitz and Meyer obtained a value for the constant $C$ for a large class of
arithmetic locally symmetric spaces, including real, complex and quaternionic hyperbolic orbifolds
\cite{LLM17}. However, the values of the constants obtained in \cite{LLM17} are not optimal, as 
the comparison with the results mentioned above shows. This 
paper is dedicated to improving the constant for quaternionic hyperbolic spaces. The main result is
the following.

\begin{theorem}\label{main}
Let $M=\Gamma\backslash \HH^n$ be a compact quaternionic hyperbolic orbifold, defined over the
number field $k$.  There exists a finite set $S$ of prime ideals
of $\O_k$ such that the principal congruence subgroup
$\Gamma_I$ associated with any ideal $I$ prime to $S$ satisfies 
\begin{equation*}
    \sys_{1}(M_{I})\geq\frac{4}{(n+1)(2n+3)}\log\big(\vol(M_{I})\big)-d,
\end{equation*}
where $M_I=\Gamma_I\backslash\HH^n$ and $d$ is a constant independent of $I$.
\end{theorem}

%The conditions on $I$ depend only on the arithmetic invariants defining $M$, see Theorem
%\ref{maintheorem} for the precise statement. 
The proof of Theorem \ref{main} appears at the end of Section \ref{sec:lowerbound}, after the needed preparation. 
Let us point out that the definition of $\Gamma_I$ depends on the choice of an embedding; however
the result of Theorem \ref{main} is not affected by this choice.  For concrete $\Gamma$ the set of
primes $S$ can be made explicit; see Remark~\ref{rmk:bad-places}.

Note that when $n=1$, $\textbf{H}_{\mathbb{H}}^{1}$ is isometric
to the four dimensional real hyperbolic space, and the constant $\frac{2}{5}$ agrees with that of
\cite{Murillo}. In Section~\ref{optimal} we generalize the argument of  D\'oria and Murillo to prove that the constant
$\frac{4}{(n+1)(2n+3)}$ is optimal; see Theorem \ref{thm:optimal}. 

%This article is organized as follows. In Section \ref{Quaternionic hyperbolic $n$-space} we recall
%basic notions of the geometry of the quaternionic hyperbolic space. In Section \ref{Eigenvalues and
%translation length}, the existence of right eigenvalues of elements in $\Sp(n,1)$ and its relation
%with the translation length is developed. We will introduce arithmetic subgroups of $\Sp(n,1)$ in
%Section \ref{sub:arithmeticity_of_lattices}, and determine a lower bound for the translation length of elements
%in the principal congruence subgroups in Section \ref{lowerbounc}. 
%Sections \ref{sub:localizations}, \ref{sub:an_upper_bound_for_the_index} are
%devoted to the technical part of relating the index of the principal congruence subgroup of an
%arithmetic group to norm of the ideal, and the proof of the results mentioned above are presented in
%Section \ref{Proof of the main results}.
%%%%%%%%%%%%%%%%%%%%%%%%%%%%%%%%%%%%%%%%%%%%%%%%%%%%%%%%%%%%%%%%%%%%%%%%%%%%%%%%%%%%%%%%%%%

\section{The quaternionic hyperbolic space}
\label{Quaternionic hyperbolic $n$-space}

\subsection{Hamiltonian quaternions}%
\label{sub:quaternions}

Let $\mathbb{H}$ be the $\mathbb{R}$-algebra of the Hamilton quaternions
$q=q_{0}+q_{1}i+q_{2}j+q_{3}ij$ with $q_{i}$ being real numbers, and $i^2=-1$, $j^2=-1$, $ij=-ji$.
Any quaternion $q$ has a {\em conjugate} $\overline{q}=q_{0}-q_{1}i-q_{2}j-q_{3}ij$, and the {\em norm} of $q$ is
given by $|q|=\sqrt{ q \overline{q}}$. The {\em real part} of $q$ is $\R(q)=q_{0}$,
and its {\em imaginary part} is $\im(q)=q_{1}i+q_{2}j+q_{3}ij$. We will consider the field $\mathbb{C}$ of
complex numbers as the subring of $\H$ consisting of the quaternions of the form $q_{0}+q_{1}i$. We say
that two quaternions $p, q$ are \emph{similar} if there exist a non-zero $r\in\H$ such that
$q=rpr^{-1}.$

\begin{lemma}\label{similarity}
Any quaternion $q$ is similar to a complex number with the same norm and the same real part. 
\end{lemma}

\begin{proof}
Let $p=\R(q)+|\im(q)|i$. If we take $r=\im(q)+|\im(q)|i$, a direct computation shows that $rpr^{-1}=q$.
\end{proof}

\subsection{Matrices over $\H$}%
\label{sub:matrices_over_H}

Any matrix $A$ with coefficients in $\mathbb{H}$ has a conjugate $\overline{A}$ whose entries are
the conjugates of the corresponding entries of $A$, and has a transpose $A^{t}$ which columns
corresponds to the rows of $A$ in the classical way. It is clear that
$(\overline{A})^{t}=\overline{(A^{t})}$, and we denote that matrix by $A^{*}$. 
For matrices $A$ and $B$ of suitable sizes 
we can verify that $(AB)^{*}=B^{*}A^{*}$, as in the complex case.
A square matrix with coefficients in $\mathbb{H}$ is called \emph{hermitian} if $A^{*}=A$. 
It is \emph{unitary} if $AA^{*}=A^{*}A=I$, where $I$ is the identity matrix.

\subsection{A model for $\HH^n$}%
\label{sub:subsection_name}

We denote by $\H^{n,1}$ the right $\H$-module $\H^{n+1}$ equip\-ped 
with the standard hermitian product of signature $(n,1)$, given by
%$$ \langle \textbf{x}, \textbf{y} \rangle=  \overline{x}_1y_1+\cdots+ \overline{x}_ny_n -\overline{x}_{n+1}y_{n+1}.$$
$$ \langle \textbf{x}, \textbf{y} \rangle= -\overline{x}_{0}y_{0}+\overline{x}_1y_1+\cdots+ \overline{x}_ny_n.$$
We consider the subspaces 
\begin{align*}
    V_{-}&=\{\textbf{x}\in \H^{n,1}\;|\;\langle \textbf{x,x}\rangle<0\},\\
    V_{0}&=\{\textbf{x}\in \H^{n,1}\smallsetminus\{0\}\;|\;\langle \textbf{x,x}\rangle=0\},\\
    V_{+}&=\{\textbf{x}\in \H^{n,1}\;|\;\langle \textbf{x,x}\rangle>0\},
\end{align*}
and the following map:
\begin{align*}
        P \colon \H^{n,1}\smallsetminus\{{\bf z}=(z_0,\cdots,z_{n})\;|\;z_{0}=0\} &\to \H^{n}\\
  \textbf{z}=\begin{pmatrix}
  z_{0}\\
  z_{1}\\
  \vdots\\
  z_{n}
  \end{pmatrix}
  \mapsto \begin{pmatrix}
    z_{1}z_{0}^{-1}\\
  z_{2}z_{0}^{-1}\\
  \vdots\\
  z_{n}z_{0}^{-1}
  \end{pmatrix}.
\end{align*}
It is clear that $P(\textbf{z})=P(\textbf{w})$ if and only if $\textbf{w}=\textbf{z}\lambda$ for some $\lambda\in\H$. 

The \textit{quaternionic hyperbolic $n$-space} is defined as
$\textbf{H}_{\mathbb{H}}^{n}=PV_{-},$ and its ideal boundary is
$\partial\textbf{H}_{\H}^{n}=PV_{0}.$ The hermitian product induces a
Riemannian metric in $\textbf{H}_{\mathbb{H}}^{n}$ given by (see \cite{KP}):
\begin{equation*}
    ds^{2}=\frac{-4}{\langle\textbf{z},\textbf{z}\rangle^{2}}\det
    \begin{pmatrix}
    \langle\textbf{z},\textbf{z}\rangle & \langle\textbf{z},d\textbf{z}\rangle\\
    \langle d\textbf{z},\textbf{z}\rangle & \langle d\textbf{z},d\textbf{z}\rangle\\
    \end{pmatrix}.
\end{equation*}
This metric is normalized so that the sectional  curvature is pinched between $-1$ and $-\frac{1}{4}$.
The distance function $\rho(\cdot,\cdot)$ in $\textbf{H}_{\mathbb{H}}^{n}$ induced by this Riemannian metric
satisfies the formula
\begin{equation}\label{distanceformula}
    \cosh^2{\left(\frac{\rho(z,w)}{2}\right)}=\frac{\langle\textbf{z},\textbf{w}\rangle\langle\textbf{w},\textbf{z}\rangle}{\langle\textbf{z},\textbf{z}\rangle\langle\textbf{w},\textbf{w}\rangle},
\end{equation}
where $\textbf{z}, \textbf{w}$ are any preimages under $P$ of $z$ and $w$  respectively. 

\subsection{The isometry group}%
\label{sub:the_isometry_group}

The set of invertible right $\H$-linear transformations of $\H^{n+1}$ identifies with the set of
invertible $(n+1)\times(n+1)$ matrices with entries in $\H$, denoted by $\GL_{n+1}(\H)$.
Let $\Sp(n,1)$ denote the subgroup of $\GL_{n+1}(\H)$ that preserves the hermitian form $\langle
\cdot, \cdot \rangle$. Equivalently,
\begin{equation*}
    \Sp(n,1)=\{A\in\GL_{n+1}(\H)| A^*JA=J\},
\end{equation*}
where $J=\text{diag}(-1,1,\dots,1)$. The elements $A \in \Sp(n,1)$ act on $\textbf{H}_{\H}^{n}$ as
follows: $A(P(\textbf{w}))=P(A(\textbf{w}))$. This action preserves the Riemannian structure on
$\textbf{H}^n_{\H}$. In fact, the isometry group $\Isom(\textbf{H}_{\H}^{n})$ is isomorphic to the
quotient $\PSp(n,1)=\Sp(n,1)/\{\pm I\}$.

\section{Eigenvalues and translation lengths in $\Sp(n,1)$}\label{Eigenvalues and translation length}

%For a locally symmetric space, it is well known that the length of a closed geodesic is related to
%the eigenvalues of the corresponding element in its fundamental group. 
This section studies the algebraic properties of the eigenvalues in the quaternionic case
(Section~\ref{sub:eignenvalues_of_quaternionic_matrices}), 
and important implications for the translation lengths in $\HH^n$ (Proposition \ref{keytheorem}).

%For instance, when the space
%is real hyperbolic, such eigenvalues contain not only information of the length but also the
%rotational angles of loxodromic elements. In the quaternionic hyperbolic case, the algebraic
%properties of the eigenvalues, although may be known among experts, are not available in the
%literature. 

%Due to its importance for the purpose of this article, and the interest in its own we
%dedicate this section to prove these properties.

\subsection{Eigenvalues of quaternionic matrices}%
\label{sub:eignenvalues_of_quaternionic_matrices}

For matrices with coefficients in a general ring there is no theory of eigenvalues. However, for
division algebras there is a chance of developing this theory, and in the case of $M_{n}(\H)$ we can
trace back this to the work of C. Lee in the late 40s \cite{Lee49}. 

\begin{definition}\label{eig}
        Let $A$ be an element in $M_{n}(\H)$. An {\em (right) eigenvalue} of $A$ is a complex number $t$ such that 
\begin{equation}\label{eigenvalue}
Av=vt    
\end{equation}
for some nonzero vector  $v$ in $\H^{n}.$
\end{definition}

\begin{remark}\label{remark} 
        If $Av=vt$ for a quaternion number $t$, by Lemma \ref{similarity} we
        can find a quaternion $\lambda$ such that  $\lambda t \lambda^{-1}$ is a complex number. Then
        $Av\lambda^{-1}=v\lambda^{-1} \lambda t \lambda^{-1}$, which shows that $\lambda t
        \lambda^{-1}$  is a complex eigenvalue.
\end{remark}

Let $A$ be an element in $M_{n}(\H)$. We can write $A=A_{1}+jA_{2},$ where $A_{1}, A_{2}\in M_{n}(\mathbb{C})$, and consider the map 
\begin{align}\label{map}
\begin{split} 
        f:\; & M_{n}(\H) \rightarrow M_{2n}(\mathbb{C})\\
          & A \mapsto
        \begin{pmatrix}
          A_{1} &-\overline{A_{2}}\\
          A_{2}  &\phantom{-}\overline{A_{1}}\\
        \end{pmatrix}.
\end{split}
 \end{align}
 The next theorem summarizes Lee's results that are relevant to us. Eigenvalues are counted {\em
 with} multiplicities. 
\begin{theorem}[Lee]
        \label{iso }
        The map $f$  is an isomorphism of rings from $M_{n}(\mathbb{H})$ into its image in $M_{2n}(\mathbb{C})$.
        Moreover for any $A \in M_n(\H)$ we have:
        \begin{enumerate}
                \item the eigenvalues of $A$ corresponds exactly to those of $f(A)$, 
                        and they fall into $n$ pairs of complex conjugate numbers; 
                \item $f(A^{*})=f(A)^{*}$.
        \end{enumerate}
\end{theorem}  

\begin{proof}
        See \cite{Lee49}:  Section~4, and Theorems 2 and 5.
\end{proof}

%\begin{proposition}\label{conjugateseigenvalues}
%Any matrix $A\in \GL_{n}(\H)$ has exactly $2n$ eigenvalues. Those eigenvalues fall into $n$ pairs, each pair consisting of two conjugate complex numbers. 
%\end{proposition}

%\begin{proof}

%It is possible to prove that $t$ is an eigenvalue of $A$ if and only if $t$ is an eigenvalue of $f(A)$ \cite[Sec. 2]{Lee49}. Since $\mathbb{C}$ is algebraically closed any $n\times n$ matrix with quaternion coefficients has exactly $2n$ eigenvalues. The fact that the eigenvalues of $f(A)$ fall into pairs of conjugate complex numbers follows from \cite[Sec. 5]{Lee49}. 
%\end{proof}

\begin{corollary}\label{traceeigenvalues}

        Let $t_{1},\overline{t}_{1},\ldots,t_{n}, \overline{t_{n}}$ denote the eigenvalues of $A \in
        M_n(\H)$. Then for the trace the following holds:
        %For any $n\times n$ matrix with quaternion coefficients and eigenvalues $t_{1},\overline{t}_{1},\ldots,t_{n}, \overline{t_{n}}$ the following relation holds 
        \begin{equation}
          \R(\tr(A))= \frac{t_{1}+\overline{t}_{1}+\cdots+t_{n}+\overline{t_{n}}}{2}.
        \end{equation}
        In particular, $\R(\tr(A))$ is a conjugation invariant.
\end{corollary}

\begin{proof}
A direct computation shows that $\R(\tr(A))=\frac{\tr(f(A))}{2}$, so that the result follows from
Theorem \ref{traceeigenvalues}.
%Now, since $t_{1},\overline{t}_{1},\ldots,t_{n}, \overline{t_{n}}$ are precisely the eigenvalues of $f(A)$, and $f$ is an isomorphism into $M_{n}(\mathbb{C})$ then \eqref{traceeigenvalues} holds, and this quantity is conjugation invariant.

\end{proof}

\begin{corollary}\label{eigenvalues}
%For any $n\times n$ matrix $A$ with quaternion coefficients, the eigenvalues of $A$ and $A^*$ are the same.
        The eigenvalues of $A$ and $A^*$ coincide.
\end{corollary}

\begin{proof}
Since $f(A^{*})=f(A)^*$ the eigenvalues of $A^*$ are the eigenvalues of $f(A)^*$, which are the complex conjugates of the eigenvalues of $f(A)$. 
But the set of these eigenvalues is invariant by complex conjugation.
%The result now follows from Proposition \ref{conjugateseigenvalues}.
\end{proof}

\begin{corollary}\label{unitary}
If $A\in M_{n}(\H)$ is unitary $(A^*A=I)$, then the eigenvalues of $A$ have all norm equal to $1$.
\end{corollary} 
 
 \begin{proof}
         If $A$ is unitary, then $f(A)$ is a unitary matrix in $M_{2n}(\mathbb{C})$ by
         Theorem~\ref{iso } (2). It is known that eigenvalues of unitary complex matrices have norm
         equal to $1$.
 \end{proof}

\subsection{Translation lengths in $\Sp(n,1)$}
\label{sub:translation_lengths_in_SP} 

We start with the following easy observation.
\begin{lemma} \label{lemma:t-inverse}
        Let $A \in \Sp(n,1)$. If $t \in \C$ is an eigenvalue of $A$ then so is $t^{-1}$.
\end{lemma}
\begin{proof}
      If $t$ is an eigenvalue of $A$ then clearly $t^{-1}$ is an eigenvalue of
      $A^{-1}$. Now, the equation $A^*JA=J$ implies $A^{-1}=J^{-1}A^*J$, and thus $A^{-1}$ has
      the same eigenvalues as $A^*$. But these are the same as the eigenvalues of $A$ by Corollary
      \ref{eigenvalues}. It follows that $t^{-1}$ is an eigenvalue of $A$.
\end{proof}

We can now prove the main result of this section. For the results concerning the algebra of
hermitian spaces we refer the reader to \cite{Scharlau12}.
 
\begin{proposition}
Let $A \in \Sp(n,1)$, and assume that $A$ leaves invariant a geodesic of $\HH^n$.
\begin{enumerate}\label{keytheorem}
     %\item If $t$ is an eigenvalue of $A$, $t^{-1}$ is also an eigenvalue of $A$.
     \item There are exactly $4$ eigenvalues of $A$ with norm different from $1$.
             If $t$ is one of such eigenvalues, the other such eigenvalues are given by $\overline{t}, t^{-1}$ and $\overline{t}^{-1}.$
     \item Assume that $|t|>1$ in (1). Then, the translation length $\ell_{A}$ of $A$ along the geodesic satisfies the equation $$\ell_{A}=2\ln(|t|).$$
  \end{enumerate}
\end{proposition}

\begin{proof}
Being of real rank $1$, the Lie group $\Sp(n,1)$ acts transitively on the set of geodesics of
its associated symmetric space $\HH^n$. Therefore, after conjugation, we may assume  
that $A$ fixes the geodesic curve  $\alpha(s)=(\tanh{s},0\ldots,0)$.
%Since the real rank of the Lie group $\Sp(n,1)$ is one,
%for any two geodesics in $\textbf{H}_{\H}^{n}$ there exist $P\in\Sp(n,1)$ which sends one geodesic
%to the other. Therefore, after conjugation, we can assume that $A$ fixes the curve
%$\alpha(s)=(0,\ldots,0,\tanh{s})$, which  is a geodesic in $\textbf{H}_{\H}^{n}$. 
Let $\textbf{v}_{0}=(1,-1,\ldots, 0)$ and $\textbf{v}_{1}=(1,1,0,\ldots, 0)$, 
so that $P(\textbf{v}_{0})$ and $P(\textbf{v}_{1})$ are the 
limit points of $\alpha$ in $\partial\textbf{H}_{\H}^{n}$. In particular those limit points are fixed
by $A$, i.e., there exist $\lambda, \beta\in\H$ such that 

\begin{align*}
     A\textbf{v}_{0}&=\textbf{v}_{0}\lambda, \hspace{1cm}\\ A\textbf{v}_{1}&=\textbf{v}_{1}\beta.
\end{align*}
In terms of the standard basis $\textbf{e}_{0},\ldots,\textbf{e}_{n}$ of $\H^{n,1}$ we obtain:

\begin{align}
        A\textbf{e}_{0}&=\textbf{e}_{0}\frac{\beta+\lambda}{2}+\textbf{e}_{1}\frac{\beta-\lambda}{2};\\
        A\textbf{e}_{1}&=\textbf{e}_{0}\frac{\beta-\lambda}{2}+\textbf{e}_{1}\frac{\beta+\lambda}{2}
        \label{lastcolumn}.
\end{align}
It also shows that the right $\H$-submodule $V$ of $\H^{n,1}$ generated by
$\{\mathbf{v}_{0},\mathbf{v}_{1}\}$ coincides with the right $\H$-submodule of
$\H^{n,1}$ generated by $\{\mathbf{e}_{0},\mathbf{e}_{1}\}$, and $A$ leaves $V$
invariant. Since the hermitian form restricted to $V$ has signature $(1,1)$,
its restriction to $V^\perp$ has signature $(n-1,0)$. Now, the fact that $A$
preserves the hermitian form implies that its restriction to the submodule
$V^\perp$ is a unitary transformation. Since $\textbf{e}_{2},\ldots,
\textbf{e}_{n}$ are orthogonal to $\textbf{e}_{0}$ and $\textbf{e}_{1}$ (with
respect to the hermitian form $J$), they generate $V^{\perp}$, and in a
suitable basis the matrix $A$ has the form

%\begin{equation*}
%A=\begin{pmatrix}
%B & 0 & 0 \\
%0 & \lambda & 0\\
%0 & 0 & \beta\\
%\end{pmatrix}
%\end{equation*}

\begin{equation*}
A=\begin{pmatrix}
\lambda & 0 & 0 \\
0 & \beta  & 0\\
0 & 0 & B
\end{pmatrix}
\end{equation*}
where $B \in M_{n-1}(\H)$ is unitary. By Corollary \ref{unitary} the eigenvalues of $B$ have norm
$1$; it thus remains to study the eigenvalues of the matrix
%and we obtain that the eigenvalues of $A$ with norm different from one are eigenvalues of the matrix 
\begin{equation}\label{submatrix}
\begin{pmatrix}
 \lambda & 0\\
 0 & \beta\\
\end{pmatrix}.
\end{equation}

%Since $A$ preserves the geodesic $\alpha(s)=(0,\ldots,\tanh{s})$, and
%$\alpha(0)=P(\textbf{e}_{n+1})$, it follows that  $A(P(\textbf{e}_{n+1})) =(0,\ldots,\tanh{s_{0}})$
%for some real number $s_{0}>0$. Moreover, the equality
%$$(0,\ldots,\tanh{s_{0}})=P((0,\ldots,\sinh{s_{0}}, \cosh{s_{0}}))$$ implies the existence of
%$w\in\H$ such that

Since $A$ preserves the geodesic $\alpha(s)=(\tanh{s},0,\ldots,0)$, and
$\alpha(0)=P(\textbf{e}_{0})$, it follows that  $A(P(\textbf{e}_{0})) =(\tanh{s_{0}},\ldots,0)$
for some real number $s_{0}>0$. Moreover, the equality
$$(\tanh{s_{0}},0,\ldots,0)=P((\cosh{s_{0}}, \sinh{s_{0}},0,\ldots,0))$$ implies the existence of
$w\in\H$ such that

%\begin{equation*}
%A(\textbf{e}_{n+1})=
%\begin{pmatrix}
%0\\
%\vdots\\
%0\\
%\sinh{s_{0}}\\
%\cosh{s_{0}}
%\end{pmatrix}w. 
%\end{equation*}

\begin{equation*}
A(\textbf{e}_{0})=
\begin{pmatrix}
\cosh{s_{0}}\\
\sinh{s_{0}}\\
\vdots\\
0\\
\end{pmatrix}w. 
\end{equation*}

Since $A$ preserves the hermitian form, $w$ has quaternion norm equal to $1$. Comparing with \eqref{lastcolumn} we obtain the equations
\begin{align*}
        \frac{\beta-\lambda}{2}=\sinh{s_{0}}w \quad&\mbox{and}\quad \frac{\beta+\lambda}{2}=\cosh{s_{0}}w,
\end{align*}
%\begin{align*}
%    \beta &=e^{s_{0}}w,\\ 
%    \lambda &= e^{-s_{0}}w.
%\end{align*}
which imply that $\beta =e^{s_{0}}w$, and $\lambda = e^{-s_{0}}w$. In particular, $\lambda$ and
$\beta$ have norm different from 1, and they satisfy the relations $\overline{\lambda}=\beta^{-1}$,
$\overline{\beta}=\lambda^{-1}.$ Since quaternions are similar to complex numbers with the same real
part and norm (Lemma \ref{similarity}), we may  assume that $\lambda$ and $\beta$ are complex
numbers, and then eigenvalues of $A$. See Remark \ref{remark}. Moreover, similarity in $\H$ commutes
with conjugation and preserves the norm, and then we obtain that the eigenvalues of the matrix
\eqref{submatrix} have the form $t$, $t^{-1}$, $\overline{t}$, $\overline{t}^{-1}$ with $t \in \C$
such that $|t|>1$.

%The translation length $\ell_{A}$ of $A$ is equal to the distance
%between $P(\textbf{e}_{n+1})$ and $P(A\textbf{e}_{n+1}).$ Using the distance formula
%\eqref{distanceformula}, a direct computation shows that 
%\begin{equation*}
%\ell_{A}=\rho(P(\textbf{e}_{n+1}),A(P(\textbf{e}_{n+1})))=2s_{0}=2\log(|\beta|)=2\log(|t|),
%\end{equation*}
%proving the second statement.
The translation length $\ell_{A}$ of $A$ is equal to the distance
between $P(\textbf{e}_{0})$ and $P(A\textbf{e}_{0}).$ Using the distance formula
\eqref{distanceformula}, a direct computation shows that 
\begin{equation*}
\ell_{A}=\rho(P(\textbf{e}_{0}),A(P(\textbf{e}_{0})))=2s_{0}=2\log(|\beta|)=2\log(|t|),
\end{equation*}
proving the second statement.
\end{proof}

\begin{corollary}\label{lenghttrace}
Let $A\in\Sp(n,1)$ leaving invariant a geodesic in $\HH^{n}$. 
Then $$e^{\frac{\ell_{A}}{2}}\geq\frac{\lvert\R(\tr(A))\rvert}{n+1}.$$
\end{corollary}

\begin{proof}
Let $t_{1}, \overline{t}_{1},\ldots,t_{n+1}, \overline{t}_{n+1}$ be the eigenvalues of $A$.
Assume that $t_{1}$ (or $\overline{t}_{1}$) is the eigenvalue with largest norm. By Proposition
\ref{keytheorem} we have $\ell_{A}=2\log(|t_{1}|)$, and applying Corollary \ref{traceeigenvalues} we
obtain
\begin{align*}
    \frac{\lvert\R(\tr(A))\rvert}{n+1}&=\frac{\lvert t_{1}+\overline{t}_{1}+\cdots+t_{n+1}+\overline{t}_{n+1}\rvert}{2(n+1)}\\
    &\leq\frac{\lvert t_{1}\rvert+\lvert\overline{t}_{1}\rvert+\cdots+\lvert t_{n+1}\rvert+\lvert\overline{t}_{n+1}\rvert}{2(n+1)}\\
    &\leq\lvert t_{1}\rvert\\
    &=e^{\frac{\ell_{A}}{2}}.
\end{align*}
\end{proof}

\section{Arithmetic subgroups of $\Sp(n,1)$}\label{arithmeticgroups}

\subsection{Arithmeticity of lattices}%
\label{sub:arithmeticity_of_lattices}

Combined work of Margulis, Corlette, and Gromov-Schoen shows that any lattice $\Gamma$ in
$\Sp(n,1)$ is arithmetic.  That is, there exists a number field $k$ with degree $[\mathbb Q: k]=d$ and an
absolutely simple algebraic $k$-group $\textbf{G}$ such that 
\begin{equation}\label{admisibility}
    \textbf{G}(k\otimes_{\mathbb{Q}}\mathbb{R})\cong \Sp(n,1)\times K,
\end{equation}
where $K$ is a compact group, and $\Gamma$ is commensurable with
$\textbf{G}(\mathcal{O}_{k})=\mathbf{G}\cap\GL_{m}\left (\mathcal{O}_{k}\right)$
for a fixed embedding $\mathbf{G}\rightarrow\GL_{m}$. A group $\G$ satisfying \eqref{admisibility} is
called \emph{admissible}.

The condition \eqref{admisibility} implies that $k$ is a totally real number field, $\textbf{G}$ is
a simply connected algebraic $k$-group of type $\mathrm{C}_{n+1}$, and by fixing an embedding $k\subset
\mathbb{R}$ we may assume that $\textbf{G}(\mathbb{R})=\Sp(n,1).$ Moreover, $\Gamma$ is cocompact if
and only if $k \neq \Q$ (see \cite[Prop.~2.8]{EK18}).

\subsection{Admissible groups}%
\label{sub:admissible_groups}

By the classification of simple algebraic groups any admissible $k$-group $\G$ 
is isomorphic to a unitary group $\U(V,h)$, where $D$ is a quaternion
algebra over $k$, and $V$ is the right $D$-module $D^{n+1}$ equipped with a nondegenerate hermitian form $h$
which is sesquilinear with respect to the standard involution of $D$.  
More precisely, we have the following.
\begin{proposition} \label{lem:admissible-a}
        Let $k \subset \mathbb{R}$ be a totally real number field. Any admissible $k$-group $\G$ is of the form
        $\U(V,h_a)$, where $a \in \O_k$ and
        \begin{align}\label{hermitianform}
                 h_a(x,y)=-a \overline{x}_{0}y_{0}+\sum_{i=1}^{n} \overline{x}_{i}y_{i}.
        \end{align}
\end{proposition}
\begin{proof}
It follows from
\cite[Prop.~2.6]{EK18} that $\G$ depends only on $D$, and not on the choice of $h$ provided the latter has the
correct signature over the different embedding $\sigma: k \to \RR$. Then it suffices to choose a
positive $a \in \O_k$ with $\sigma(a)<0$ for all nontrivial embeddings $\sigma: k \to \RR$; such an
element exists by weak approximation (see \cite[Sect.~1.2.2]{PR94}).
\end{proof}

\subsection{A group scheme structure}%
\label{sub:a_group_scheme_structure}

We fix an order $\mathcal{O}_{D}$ of the quaternion algebra $D$, and consider the lattice $L =
\O_D^{n+1}$ in $V$. Choosing $h$ as in \eqref{hermitianform} we obtain a hermitian module $(L,h)$ 
(more generally it suffices to take $h$ with integral coefficients). 
For any ring extension $\O_k \to A$ we consider the unitary group
\begin{align}
        \label{eq:gp-scheme}
        \G^L(A) &= \U(L \otimes_{\O_k} A, h).
\end{align}
Then $\G^L$ defines an affine group scheme over $\O_k$ with generic fiber $\G$; it is a closed subgroup
of the group scheme $\mathfrak{End}_{\O_D}(L)^\times$ defined by taking invertible endomorphisms
(see \cite[Sect.~II.2.6]{DG}). In particular the arithmetic subgroup $\Gamma = \G^L(\O_k)$ can be seen as a subgroup of the matrix
group $\GL_{n+1}(\O_D)$.

For any ideal $I \subset \O_k$ we define the principal congruence subgroup
$\Gamma_I$ as the kernel of the natural map $\G^L(\O_k) \to \G^L(\O_k/I)$. 

\begin{proposition} \label{prop:mod-I-as-quaternions}
        The subgroup $\Gamma_I$ corresponds to the kernel of the map $\Gamma \to
        \GL_{n+1}(\O_D/I\O_D)$. 
\end{proposition}
\begin{proof}
        In view of the definition \eqref{eq:gp-scheme} this follows directly from the isomorphism
        $\O_D \otimes_{\O_k} \O_k/I \cong \O_D/I\O_D$.
\end{proof}

\subsection{Localizations}%
\label{sub:localizations}

We will denote by $\V = \Vf \cup \Vi$ the set of (finite and infinite) places of $k$. 
For any $v\in\V$ the symbol $k_{v}$ denotes the completion of $k$ with respect to $v$, and
$D_{v}=D\otimes_{k}k_{v}$.  It follows from the admissibility of $\G = \U(V, h)$ that $D_v \cong \H$
for each $v \in \Vi$; see \cite[Cor.~2.5]{EK18}.
For $v \in \Vf$ we denote by $\o_{v}$ the valuation ring of $k_{v}$, and by $\pi_{v}$ a uniformizer in $\o_v$.

For $L = \O_D^{n+1}$ and $e\in \mathbb{N}$ we set $P^{(e)}_v = \ker(\G^L(\o_v) \to
\G^L(\o_v/\pi_v^e))$. Note that $P_v^{(0)} = \G^L(\o_v)$, and the latter is a hyperspecial
parahoric subgroup of $\G(k_v)$ for all but finitely many $v \in \Vf$; see \cite[Sect.~3.9.1]{Tits}.
We shall use the notation $P_v = P_v^{(0)}$.
The group $\Gamma = \G^L(\O_k)$ can thus be written as $\G(k) \cap \prod_{v\in\Vf} P_v$.
Let $I \subset \O_k$ be an ideal with prime factorization $I = \prod_{v\in\Vf} \p_v^{e_v}$.
Then it follows from the Chinese reminder theorem that 
\begin{align}
        \label{eq:localization}
       \Gamma_I &= \G(k) \cap \prod_{v \in \Vf} P_v^{(e_v)}.
\end{align}
The following lemma is a well-known consequence of the strong approximation property for $\G$. 
\begin{lemma} \label{lem:index-str-ap}
        For the index the following equality holds:
        \begin{align*}
                [\Gamma:\Gamma_I] =  \prod_{v \in \Vf} [P_v:P_v^{(e_v)}]. 
        \end{align*}
\end{lemma}
\begin{proof}
   For two subgroups $A, B$ of a common group there is a bijection between $A/A\cap B$ and $A B/B$. 
   The result follows with $A = \Gamma = \G(k) \cap \prod P_v$ and $B = \prod P_v^{(e_v)}$, noting that $\G(k) B$ is 
   the whole adelic group $\G(\mathbb{A}_\mathrm{f})$ by strong approximation. 
\end{proof}

\begin{lemma} \label{lem:index-hyperspecial} Assume that $P_v = \G^L(\o_v)$ is parahoric hyperspecial. Then 
        \begin{align*}
                [P_v:P_v^{(e)}] &= q_v^{e (n+1) (2n+3)} \prod_{j=1}^{n+1} \left(1 -
                        \frac{1}{q_v^{2j}}
                \right), 
        \end{align*}
        where $q_v$ denotes the order of the residue field $\mathfrak{f}_v = \o_v/\pi_v$.
\end{lemma}
\begin{proof}
        If $\G^L(\o_v)$ is hyperspecial then $\G_{k_v}$ must be split, and $\G^L_{\o_v}$ is the
        Chevalley group scheme of type $\mathrm{C}_{n+1}$ (see \cite[Sect.~3.4.2]{Tits}). In
        particular $\G^L_{\o_v}$ is smooth, and the reduction map $P_v \to \G^L(\o_v/\pi_v^e)$ is
        surjective. In case $e = 1$ the index is thus given by the order of $\G^L(\mathfrak{f}_v)$,
        which can be found for instance in \cite[Table 1]{Ono}. For $e>1$ this order must be
        multiplied by the order of $\ker(\G^L(\o_v/\pi_v^e) \to \G^L(\mathfrak{f}_v))$, which by
        induction over $e$ equals  $q_v^{(e-1) \dim(\G)}$. For the type $\mathrm{C}_{n+1}$ we have
        $\dim(\G) = (n+1)(2n+3)$, and the result follows. 
\end{proof}

\subsection{An upper bound for the index}%
\label{sub:an_upper_bound_for_the_index}

We keep the notation introduced above. We denote by $\N(I)$ the norm of an ideal $I \subset \O_k$,
i.e., the order of $\O_k/I$.

\begin{proposition} \label{prop:bound-index}
        Let $\Gamma = \G^L(\O_k)$.  There exists a finite set $S \subset \Vf$ such
        that for any ideal $I \subset \O_k$ prime to $S$ the following holds:
        \begin{align*}
                [\Gamma: \Gamma_I] &\le \N(I)^{(n+1)(2n+3)}. 
        \end{align*}
\end{proposition}
\begin{proof}
        Let $S$ be the set of places $v$ such that $P_v = \G^L(\o_v)$ is not hyperspecial. 
        Let $I = \prod_{v \in \Vf} \p_v^{e_v}$ be any ideal with $e_v = 0$ for each $v \in S$. 
        From \eqref{eq:localization}  and Lemma~\ref{lem:index-hyperspecial} we 
        obtain 
        \begin{align*}
                [\Gamma:\Gamma_I] &= \prod_{v \in \Vf} [P_v:P_v^{(e_v)}]\\ 
                                  &\le \prod_{v \in \Vf} q_v^{e_v (n+1) (2n+3)}\nonumber.
        \end{align*}
        But the latter equals $\N(I)^{(n+1)(2n+3)}$ since $q_v = \N(\p_v)$.
\end{proof}

\begin{remark} \label{rmk:bad-places}
        Proposition \ref{prop:bound-index} holds for any arithmetic subgroup of $\G(k)$. In the case
        $\Gamma = \G^L(\O_k)$ we can have some control on the set $S$. Let us assume that $L =
        \O_D^{n+1}$ with $\O_D$ a \emph{maximal} order, and consider the integral form $h = h_a$
        given in \eqref{hermitianform}. Then it follows from \cite[Lemmas 5.1 and 5.5]{EK18} that
        $S$ can be taken to be the set of places $v$ where either
        \begin{enumerate}
                \item $D_v$ ramifies;
                \item or $\p_v$ divides the coefficient $a$.
        \end{enumerate}
\end{remark}

\section{Bounding the systole from below}
\label{sec:lowerbound}

This section deals with the computations that provide a lower bound for the trace in a congruence
subgroup. This material is then used in Section~\ref{sub:the_proof_of_theorem_main} for bounding
the systole of the corresponding manifolds, in particular for proving Theorem~\ref{main}.

%In order to bound from below the systole of $\Gamma_{I}\backslash\textbf{H}^{n}_{\H}$ we need to bound the real part of the trace of the elements in $\Gamma_{I}$ (see Corollary \ref{lenghttrace}), and this is the purpose of this section. 
We essentially keep the notation of the preceding section:
$\Gamma$ will denote the arithmetic subgroup $\textbf{G}^{L}(\mathcal{O}_k)$, where $\textbf{G}=\textbf{U}(V,h)$ 
is an admissible $k$-group with $h = h_a$ as in \eqref{hermitianform}. It will be important to work with
the matrix representation with coefficients in $D$ (the quaternion algebra over $k$). That is, we
embed $\Gamma$ as a subgroup of
\begin{align*}
        \G(k)&=\{ C \in \GL_{n+1}(D)\mid C^{*}JC=J\}, 
\end{align*}
where $J = \mathrm{diag}(-a, 1, \dots, 1)$. In particular the trace $\tr(C)$ has the same meaning as
in Section \ref{Eigenvalues and translation length}. In accordance with the notation of Sections
\ref{Quaternionic hyperbolic $n$-space}--\ref{Eigenvalues and translation length}, we use the
convention that the rows (resp.\ columns) of the matrices are indexed from $0$ to $n$.

\subsection{Two lemmas}%
\label{sub:two_lemmas}

Recall that we have fixed an embedding $k \subset \RR$, which we refer to as the {\em trivial} embedding
(or trivial archimedean place). The symbol $I_{n+1}$ denotes the identity matrix in $\GL_{n+1}(D)$.

%The unitary group $\U(V,h)$ can be represented as the subgroup of $\GL_{n+1}(D)$ given by
%$$\textbf{U}(V,h)=\{ C \in \GL_{n+1}(D)\mid C^{*}JC=J\}$$
%with
%\begin{equation}\label{matrixh}
%J =
 %\begin{pmatrix}
  %-a & 0 & \cdots & 0 \\
  %0 & 1 & \cdots & 0 \\
  %\vdots  & \vdots  & \ddots & \vdots  \\
  %0 & 0 & \cdots & 1
 %\end{pmatrix}.  
%\end{equation}

%By Proposition \ref{prop:mod-I-as-quaternions}, for any ideal $I$ in $\mathcal{O}_{k}$ the principal congruence group $\Gamma_I$ is represented by 

%$$\Gamma_{I}=\{\left(c_{ij}\right)\in\Gamma\mid c_{ii}-1\in I\mathcal{O}_{D}, c_{ij}\in I\mathcal{O}_{D}\hspace{1mm}\mbox{for}\hspace{1mm} i\neq j\}.$$

\begin{lemma}\label{trazanon}
        Assume that $k \neq \Q$. For any $C\in\Gamma$ different from $\pm I_{n+1}$, we have $\lvert\R(\tr(C))\rvert\neq n+1.$ 
\end{lemma}
%\todo[inline]{AQUI NECESITAMOS JUSTIFICAR QUE $\sum_{i=o}^{n+1}(\R(y_{i})$ ES DIFERENTE DE CERO. PARA ESO LA IDEA ES QUE LA MATRIZ $C^{\sigma}$ ES UN ELEMENTO DE $Sp(n+1)$ Y POR TANTO ES UNA MATRIZ UNITARIA, QUE ES DIAGONALIZABLE CON AUTOVALORES NUMEROS COMPLEJOS DE NORMA 1. LUEGO, SI $\sigma( \sum_{i=o}^{n+1}(\R(y_{i}))$ ES CERO TENDRIAMOS QUE $\tr(C^{\sigma})=n$ LO CUAL DEBE IMPLICAR QUE TODOS LOS AUTOVALORES DE $C^{\sigma}$ SON 1, Y POR TANTO $C^{\sigma}=id$, LO QUE ES CONTRADICCION PORQUE c NO ES LA IDENTIDAD. PARA ARGUMENTAR ESTO DEBO EXPONER O HABLAR UN POCO SOBRE LOS AUTOVALORES DE MATRICES CON COEFICIENTES EN CUATERNIOS DE HAMILTON.}
\begin{proof}
Suppose that $\lvert\R(\tr(C))\rvert= n+1$.
Since $k \neq \Q$ there exists a nontrivial embedding  $\sigma :k\rightarrow\mathbb{R}$, for which
$\lvert\R(\tr(C^{\sigma}))\rvert=n+1$.
%Then $\lvert\R(\tr(C^{\sigma}))\rvert=n+1$ for any nontrivial archimedean place $\sigma :k\rightarrow\mathbb{R}$. 
By the admissibility condition we have   $C^{\sigma} \in \Sp(n+1)$, so that $C^{\sigma}$ is unitary (in the quaternionic sense). 
Corollary \ref{unitary} implies that the eigenvalues of $C^{\sigma}$ are all complex numbers of norm equal to one. 
With  $\lvert\R(\tr(C^{\sigma}))\rvert=n+1$ it follows that these  eigenvalues are either all equal
to $1$, or all equal to $-1$.  Since unitary matrices are diagonalizable we obtain 
$C^{\sigma}=\pm I_{n+1}$, so that $C = \pm I_{n+1}$.
\end{proof}

\begin{lemma} \label{lem:2}
        Let $C=(c_{ij})$ be an element in $\G(k)$ and write $c_{ii}=1+y_{i}$.  For every nontrivial
        embedding $\sigma: k \to \RR$  and each $i=0,\dots, n$ we have $\lvert\sigma(|c_{ii}|^{2})\rvert\leq 1$ and 
        $\lvert\sigma(\R(y_{i}))\rvert\leq 2.$
\end{lemma}
\begin{proof}
The equation $C^{*}JC=J$ implies that the columns of $C$ satisfy the equations
\begin{equation}\label{equations on the matrix 3}
-a|c_{00}|^{2}+\sum_{i=1}^{n}|c_{i0}|^{2} = -a,
\end{equation}
\begin{equation}\label{equations on the matrix 2}
-a|c_{0j}|^{2}+\sum_{i=1}^{n}|c_{ij}|^{2}=1,
\mbox{ for }j=1,\ldots ,n.
\end{equation}
where $|x|^{2}=x\overline{x}$ denotes the quaternion norm of $x$ in $D$. Since all the coefficients 
$c_{ij}$ lie in $D$, the norm $|c_{ij}|^{2}$ is an element of $k.$ Let $\sigma:k\rightarrow\mathbb{R}$ 
be a nontrivial embedding. Applying $\sigma$ to \eqref{equations on the matrix 3} we obtain 
%\begin{align*}
%-\sigma(a)\sigma(|c_{1,1}|^{2})& \leq -\sigma(a)\sigma(|c_{1,1}|^{2})+ \sum_{i=2}^{n+1}\sigma(|c_{i,1}|^{2})\\ & = -\sigma(a)
%\end{align*} 
\begin{align*}
-\sigma(a)\sigma(|c_{00}|^{2}) &\leq -\sigma(a)\sigma(|c_{00}|^{2})+
\sum_{i=1}^{n}\sigma(|c_{i0}|^{2})\\ & = -\sigma(a).
\end{align*}
Hence $\lvert\sigma(|c_{00}|^{2})\rvert\leq 1$. Similarly, applying $\sigma$ to \eqref{equations on
the matrix 2} we obtain that $\lvert\sigma(|c_{ii}|^{2})\rvert\leq 1$  for $i=1,\ldots,n$. Now, if
$D=\left(\frac{\delta,\gamma}{k}\right)$ and $c=x_{0} + x_{1}i +x_{2}j +x_{3}ij$, then
$|c|^{2}=x_{0}^{2}-\delta x_{1}^ {2}-\gamma x_{2}^{2}+\delta\gamma x_{3}^{2}$. Since $D^{\sigma}$ is
a division algebra we have $\sigma(\delta)<0$ and $\sigma(\gamma)<0$, and thus
\begin{align*}
    \sigma(|c|^{2})&=\sigma(x_{0})^{2}-\sigma(\delta)\sigma(x_{1})^{2}-\sigma(\gamma)\sigma(x_{2})^{2}+\sigma(\delta)\sigma(\gamma)\sigma(x_{3})^{2}\\
    &\geq\sigma(x_{0})^{2}\\
    &=\sigma(\R(c))^2.
\end{align*}
In particular $\sigma(|c_{ii}|^{2})\leq 1$ implies $|\sigma(\R(c_{ii}))|\leq 1$, from which one
deduces $|\sigma(\R(y_{i})|\leq 2$.
\end{proof}

\subsection{Bounding the trace}%
\label{sub:bounding_the_trace}
We want to bound the trace of a congruence subgroup $\Gamma_I$, for $I \subset \O_k$ some ideal. In
the matrix representation we have the following description:  
$$\Gamma_{I}=\{\left(c_{ij}\right)\in\Gamma\mid c_{ii}-1\in I\mathcal{O}_{D}, c_{ij}\in I\mathcal{O}_{D}\hspace{1mm}\mbox{for}\hspace{1mm} i\neq j\}.$$
We recall that the element $a \in \O_k$ appears (with negative sign) as the unique nontrivial coefficient
of the hermitian form $h = h_a$ that determines $\G$.

\begin{lemma}\label{estimative of sigma}
        Let  $C\in\Gamma_{I}$, and write $c_{ii} = 1 + y_i$. Then
        \begin{equation}
                \label{eq:2a}
        2a\sum_{i=0}^{n}\R(y_{i})\in I^{2}.
        \end{equation}
\end{lemma}

%
%\begin{proof}
%The equation $C^{*}JC=J$ implies that the columns of $C$ satisfy the equations
%\begin{equation}\label{equations on the matrix 3}
%-a|c_{1,1}|^{2}+\sum_{i=2}^{n+1}|c_{i, 1}|^{2} = -a,
%\end{equation}
%\begin{equation}\label{equations on the matrix 2}
%-a|c_{1, j}|^{2}+\sum_{i=2}^{n+1}|c_{i, j}|^{2}=1,
%\mbox{ for }j=2,\ldots ,n+1.
%\end{equation}

\begin{proof}
We first replace $c_{00}=1+y_{0}$ in \eqref{equations on the matrix 3} to obtain
\begin{equation}\label{moduloI}
-a(2\R(y_{0})+|y_{0}|^{2}) +\sum_{i=1}^{n}|c_{i0}|^{2} = 0.
\end{equation}
For $C\in\Gamma_{I}$ we have  $y_{0}\in I\mathcal{O}_{D}$ and $c_{i0}\in I\mathcal{O}_{D}$ for
$i>0$. From \eqref{moduloI} it follows that $2a\R(y_{0})\in I^{2}$. By replacing
$c_{ii}=1+y_{i}$ in \eqref{equations on the matrix 2} the same argument shows that $2\R(y_{i})\in
I^{2}$ for $i>0$. Since $a \in \O_k$ we have that $2a\R(y_i) \in I^2$ for all $i = 0,\dots, n$, and
thus the same holds for their sum.
\end{proof}

In the following the symbol $\N(\cdot)$ denotes either the norm $\N_{k/\Q}(\cdot)$ for elements of $k$, or
the norm of ideals in $\O_k$. Recall that for a principal ideal $I = (\alpha)$, one has $\N(I) =
|\N(\alpha)|$ unless $\alpha = 0$.
\begin{corollary} \label{cor:bd-norm}
        Let $\Gamma_I$ be defined over the number field $k$ of degree $d>1$, and let $C \in
        \Gamma_I$ different from $\pm I_{n+1}$.   Then 
        \begin{align*}
                \left| \N\left(\sum_{i=0}^{n}\R(y_{i})\right)\right| &\geq \frac{\N(I)^{2}}{2^{d}\N(a)},
        \end{align*}
        where $c_{ii} = 1+y_i$.
\end{corollary}

\begin{proof}
        We have $\sum_{i=0}^{n}\R(y_{i})\neq 0$ by Lemma \ref{trazanon}. 
        The result follows then immediately by applying $\N(\cdot)$ on \eqref{eq:2a}.
        %$$\Bigg\lvert \N\left(\sum_{i=0}^{n}\R (y_{i})\right)\Bigg\rvert\geq \frac{\N(I)^{2}}{2^{d}\N(a)}.$$\qedhere
\end{proof}

%\begin{remark} \label{rmk:I-2}
        %We remark that $-I_{n+1} \in \Gamma_I$ only if $I  
%\end{remark}

\begin{proposition}\label{tracenorm1st}
        Let $k$ be of degree $d>1$, and $I \subset \O_k$ be a proper nontrivial ideal.
        For any $C\in\Gamma_{I}$ different from $\pm I_{n+1}$  we have $$\lvert\R(\tr(C))\rvert\geq  \frac{\N\left(I\right)^{2}}{2^{2d-1}(n+1)^{d-1}\N(a)}-n-1.$$
\end{proposition}

\begin{proof}
        %Notice that for $k\neq\Q$ we have $\N(I) > 2$, so that $C \neq - I_{n+1}$ (nor $I_{n+1}$ by
        %hypothesis). 
By Lemma \ref{lem:2} we have 
\begin{align*}
\Bigg\lvert \N\left(\sum_{i=0}^{n}\R(y_{i})\right)\Bigg\rvert &=\Bigg\lvert \sum_{i=0}^{n}\R(y_{i})\Bigg\rvert\Bigg\lvert\prod_{\sigma\neq id} \sigma \left(\sum_{i=0}^{n}\R(y_{i})\right) \Bigg\rvert \\
  &\leq\Bigg\lvert\sum_{i=0}^{n}\R(y_{i})\Bigg\rvert\cdot 2^{d-1} (n+1)^{d-1}.    
\end{align*}
With Corollary \ref{cor:bd-norm} we obtain  
$$\Bigg\lvert \sum_{i=0}^{n}\R(y_{i})\Bigg\rvert \geq \frac{\N\left(I\right)^{2}}{2^{2d-1}(n+1)^{d-1}\N(a)}.$$\\
Now, since $\R(\tr(C)) = n+1 + \sum_{i=0}^{n}\R(y_{i})$
we have 
\begin{align*}
\lvert \R(\tr(C))\rvert & \geq \Bigg\lvert \sum_{i=0}^{n}\R(y_{i}) \Bigg\rvert -n-1\\
& \geq  \frac{\N\left(I\right)^{2}}{2^{2d-1}(n+1)^{d-1}\N(a)}-n-1.\label{main computation}
\end{align*}
\end{proof}

\subsection{Bounding the systole}
\label{sub:the_proof_of_theorem_main}
We can now use the preceding results to obtain a lower bound for the systole of $\Gamma_I\backslash\textbf{H}_{\mathbb{H}}^n$ in terms of the norm of the ideal $I$.
\begin{proposition}\label{mainprop}
Let $M=\Gamma\backslash \HH^n$ be a compact arithmetic orbifold with $\Gamma=\mathbf{G}^L(\mathcal{O}_k)$. 
If $\Gamma_I$ is a principal congruence subgroup associated to an ideal $I \subset \O_k$ then
\begin{equation*}
\sys_{1}(M_I)\geq 4\log(\N(I))-c,
\end{equation*}
where $M_I=\Gamma_I\backslash\HH^{n}$ is the associated congruence cover of $M$, and $c$ is a constant independent of $I$. 
\end{proposition}

\begin{proof}
%For $\alpha$ a shortest closed geodesic in $M_I$, choose a corresponding element $A\in \Gamma_I$
%whose translation length $\ell_{A}$ equals $\sys_{1}(M_{I})$. 
Let $A \in \Gamma_I$ corresponding to a shortest closed geodesic in $M_I$, so that 
its translation length $\ell_{A}$ equals $\sys_{1}(M_{I})$.
%We note that $\R(\tr(A))$ and
%$\ell_{A}$ do not depend on this choice as those quantities are conjugacy invariants. 
By Corollary \ref{lenghttrace} and Proposition \ref{tracenorm1st} we obtain %that $A$ satisfies
\begin{align*}
    \ell_{A}&\geq 2\log\bigg(\frac{\lvert\R(\tr(A))\rvert}{n+1}\bigg)\\
            &\geq 2 \log\bigg(\frac{\N\left(I\right)^{2}}{2^{2d-1}(n+1)^{d}\N(a)}-1\bigg)\\
            &\geq 2 \log\bigg(\frac{\N\left(I\right)^{2}}{2\cdot2^{2d-1}(n+1)^{d}\N(a)}\bigg)\\
            &=4\log(\N(I))-2\log\big(2^{2d}(n+1)^{d}\N(a)\big)
\end{align*}
if $\N(I)^{2}\geq 2^{2d}(n+1)^{d}\N(a).$  Since there exist only finitely many ideals 
$I \subset \O_k$ with bounded norm, the result follows by enlarging the constant $c$ if necessary.\qedhere

\end{proof}

We can now prove the main result.
\begin{proof}[Proof of Theorem \ref{main}]
%\begin{theorem}\label{maintheorem}
%Let $M=\Gamma\backslash \textbf{H}^n_\H$ be a compact quaternionic hyperbolic orbifold defined over a totally real number field $k$. Then there exists a finite set $S$ of prime ideals of $\mathcal{O}_{k}$ depending only on $\Gamma$ such that, for any ideal $I\subset\mathcal{O}_{k}$ with no prime factors in $S$, the principal congruence subgroup $\Gamma_I$ associated to $I$ satisfies 
%\begin{equation*}
    %\sys_{1}(M_{I})\geq\frac{4}{(n+1)(2n+3)}\log\big(\vol(M_{I})\big)-d,
%\end{equation*}
%where $M_I=\Gamma_I\backslash\textbf{H}_{\H}^{n}$ is the associated congruence cover of $M$, and $d$ is a constant independent of $I$. 
%\end{theorem}
Let $M = \Gamma\backslash\HH^n$ be a compact quaternionic orbifold. 
Then $\Gamma \subset \G(k)$ for some admissible $k$-group $\G$ with $k \neq \Q$.
%We have  that $\Gamma$ is commensurable with
%$\mathbf{G}^{L}(\mathcal{O}_{k})$ for some admissible algebraic $k$-group
%$\mathbf{G}=\mathbf{U}(V,h)$, with $h$ as in \eqref{hermitianform} and $k \neq \Q$.
On the other hand, by \cite[Prop. 2.2]{LLM17} 
we can replace the study of $\Gamma$ with any subgroup commensurable with it, in particular, we may
assume that $\Gamma=\mathbf{G}^{L}(\mathcal{O}_k)$ as above. 
By Proposition~\ref{prop:bound-index} there exist a finite set $S$ of prime ideals of $\mathcal{O}_{k}$ such that
any ideal $I\subset\mathcal{O}_{k}$ with no prime factors in $S$ satisfies
$$[\Gamma:\Gamma_{I}]\leq\N\left (I\right )^{(n+1)(2n+3)}.$$  
From Proposition \ref{mainprop} we obtain 
\begin{equation*}
    \sys_{1}(M_{I})\geq\frac{4}{(n+1)(2n+3)}\log\big([\Gamma:\Gamma_{I}]\big)-c,
\end{equation*}
for some constant $c$ independent of $I$. 
The result then follows with the equality $\vol(M_I)=\vol(M) [\Gamma:\Gamma_I]$.
\end{proof}

\section{Optimality of the constant}\label{optimal}

%For $n=1$ the quaternionic hyperbolic space $\textbf{H}_{\H}^{1}$ is
%isometric to the real hyperbolic space $\textbf{H}_{\mathbb{R}}^{4}$, and according to
%\cite{Murillo} the multiplicative constant in Theorem \ref{main} in this case is optimal. It
%suggests that the constant $\frac{4}{(n+1)(2n+3)}$ is sharp in general, and it turns out that
In this section we show that the constant $\frac{4}{(n+1)(2n+3)}$ is sharp, using 
similar arguments as in the Appendix of \cite{Murillo}. The precise result is the following.

\begin{theorem}\label{thm:optimal}
%In any commensurability class of cocompact lattices in $\Sp(n,1) there $,
%defined over a totally real number field $k\neq\mathbb{Q}$. 
Let $k \subset \RR$ be totally real with $k \neq \Q$, and let $\G$ be an admissible $k$-group, so
that $\G(\RR) = \Sp(n,1)$. Then there exists an arithmetic subgroup $\Gamma \subset \G(k)$  
such that for any sequence of prime ideals $\mathfrak{p}\subset \mathcal{O}_{k}$ 
%with norm large enough 
the principal congruence
subgroups $\Gamma_{\mathfrak{p}}$ satisfy
\begin{equation*}
    \sys_{1}(M_{\mathfrak{p}})\leq\frac{4}{(n+1)(2n+3)}\log\big(\vol(M_{\mathfrak{p}})\big)+d',
\end{equation*}
where $M_{\mathfrak{p}}=\Gamma_{\mathfrak{p}}\backslash\HH^n$ and $d'$ is a constant independent of $\mathfrak{p}$.
\end{theorem}

\begin{proof}
Let $\G = \U(V,h_a)$ and set $\Gamma = \G^L(\O_k)$ (see Sections
\ref{sub:admissible_groups}--\ref{sub:a_group_scheme_structure}).
The idea  is to construct arithmetic real hyperbolic manifolds which are totally geodesic
submanifolds in $\Gamma_\p\backslash\HH^n$, and to apply \cite[Thm. A.1]{Murillo}. The latter proves
in particular the case $n=1$, so that we will assume $n>1$ hereafter. 

%Passing to a commensurable group, we can assume that $\Gamma=\textbf{G}^{L}(\mathcal{O}_k)$, where
%$\textbf{G}=\textbf{U}(V,h_a)$ is an admissible algebraic $k$-group, and in a suitable basis
%$\{e_{0}, e_{2},\ldots, e_{n}\}$ of $V$ as a right $D$-module, the hermitian form $h_a$ is given as
%in \eqref{hermitianform}.
We denote by $\{e_{0}, \ldots, e_{n}\}$ the standard basis of $V = D^{n+1}$; recall that $D$
is a quaternion algebra over $k$.
%The compactness condition implies that $k\neq\mathbb{Q}$, see \cite[Prop. 2.8]{EK18}. The
%admissibility condition implies that $a\in\mathcal{O}_k$ is positive, and for any non-trivial
%embedding $\sigma:k\rightarrow\mathbb{R}$ we have that $\sigma(a)<0$.
Let $W$ be the $k$-vector space generated by $\{e_{0}, \ldots, e_{n}\}$, and $L^{'} \subset W$ be the
$\O_k$-lattice with the same basis. 
%the $\mathcal{O}_{k}$-lattice in $W$ given by $L^{'}=\mathcal{O}_{k}^{n+1}$. 
Consider the quadratic form on $W$ given by
$$q(x,y)=-ax_{0}y_{0}+\sum_{i=1}^{n}x_{i}y_{i}.$$
This is the restriction to $W$ of the Hermitian form $h_a$; it is admissible in the sense of \cite[Sec. 2.3]{Murillo}.
We consider the  $k$-group $\textbf{SO}(W,q)$ and its
simply connected cover $\textbf{Spin}(W,q)$. The Lie group
$\textbf{Spin}(W,q)(\mathbb{R})/\lbrace{\pm I\rbrace }$ is isomorphic to the orientation preserving
isometry group of the real hyperbolic $n$-space $\textbf{H}_{\mathbb{R}}^{n}$.  
%Considering the matrix representations defined by the 
Since the basis $\{e_{0}, \ldots, e_{n}\}$ is common to  $V$ and $W$ we have an inclusion
$\textbf{SO}(W,q)\subset \textbf{U}(V,h_a)$,  and composing with the isogeny  we obtain a homomorphism
$\textbf{Spin}(W,q)\rightarrow\G$ defined over $k$.
%induced by the canonical isogeny
%$\textbf{Spin}(W,q)\rightarrow\textbf{SO}(W,q).$ 

Let $\Gamma' \subset \mathbf{Spin}(W,q)(k)$ be the stabilizer of $L'$. Since $k \neq \Q$ it is a cocompact
arithmetic lattice in $\Spin(n,1)$. The map  $\textbf{Spin}(W,q)\rightarrow\G$ induces a map
$\Gamma' \to \Gamma$, and similarly $\Gamma'_I \to \Gamma_I$ for any ideal   $I\subset\mathcal{O}_{k}$.
%By the admissibility condition, the stabilizer of $L^{'}$ in $W$, that we denote by $\Gamma^{'}$, embedds as a cocompact arithmetic lattice in $\Spin(n,1)$. It follows from the definitions that $\textbf{Spin}(W,q)\rightarrow\textbf{U}(V,h_a)$ restricts to a group homomorphism $\Gamma^{'}\rightarrow \Gamma$ and also to $\Gamma^{'}_I\rightarrow \Gamma_I$ for any ideal $I\subset\mathcal{O}_{k}$.
%Considering an isometric embedding of $\textbf{H}_{\mathbb{R}}^{n}$ to $\textbf{H}_{\H}^{n}$
%equivariant under the respective actions of $\Gamma'$ and $\Gamma$, the map $\Gamma^{'}_I\rightarrow
%\Gamma_I$
This induces a totally geodesic embedding $$T_{I}\hookrightarrow M_{I},$$ where
$T_{I}=\Gamma'_{I}\backslash\textbf{H}^{n}_{\mathbb{R}}$. Therefore
$$\sys_{1}(M_{I})\leq\sys_{1}(T_{I}).$$
%On the other hand, $T_{I}$ is a compact arithmetic real hyperbolic manifold whenever $I$ has norm large enough. 

From now on, we will assume that $I=\mathfrak{p}$ is a prime ideal. By \cite[Theorem~A.1]{Murillo}
(see also \cite[Theorem~B]{Murillothesis})
there exists a constant $d$ independent of $\mathfrak{p}$ such that
\begin{equation}\label{systolesubmanifold}
\sys_{1}(T_{\mathfrak{p}})\leq\frac{8}{n(n+1)}\log(\vol(T_{\mathfrak{p}}))+d. 
\end{equation}
%whenever the norm of $\mathfrak{p}$ is large enough (see also \cite[Thm. B]{Murillothesis}).
Following the argument as in \cite[Theorem~B]{Murillothesis}, there exist constants $a_{1}$ and $a_{2}$ such that 
%for $\mathfrak{p}$ with norm large enough  
\begin{equation}\label{firstorder}
a_{1}\leq\frac{\vol(T_{\mathfrak{p}})}{\N(\mathfrak{p})^{\frac{n(n+1)}{2}}}\leq a_{2}.  
\end{equation}
%Let $v$ denotes the place of $k$ corresponding to $\mathfrak{p}$. 
For $\mathfrak{p}$ of norm large enough we have that  $\textbf{G}^{L}(\o_\p)$ is parahoric hyperspecial. 
%Since $\mathbf{G}$ is simply
%connected and admissible, it satisfies the strong approximation theorem with respect to
%$\mathcal{V}_{k}^{\infty}$ (cf. \cite[Thm. 7.12]{PR94}). It is a well-known consequence of strong
%approximation that for all but finitely many $\mathfrak{p}$ the quotients
%$\Gamma/\Gamma_\mathfrak{p}$ and $P_v/P_v^{(1)}$ are isomorphic.
By Lemmas \ref{lem:index-str-ap} and \ref{lem:index-hyperspecial} we obtain 
$$[\Gamma:\Gamma_\mathfrak{p}]=\N(\mathfrak{p})^{(n+1)(2n+3)}\prod_{j=1}^{n+1}\left(1-\frac{1}{\N(\mathfrak{p})^{2j}}\right).$$
Since $\vol(M_\mathfrak{p})=\vol(M)[\Gamma:\Gamma_\mathfrak{p}]$ there exist positive constants $b_{1}$ and $b_{2}$ such that 
\begin{equation}\label{secondorder}
    b_{1}\leq\frac{\vol(M_{\mathfrak{p}})}{\N(\mathfrak{p})^{(n+1)(2n+3)}}\leq b_{2}.
\end{equation}
%when $\N(\mathfrak{p})$ is large enough to guarantee that $\left (1-\frac{1}{\N(\mathfrak{p})^{2j}}\right )\sim 1$ for $j=1,\cdots,n+1$. 

By plugging the right-hand side of \eqref{firstorder} in \eqref{systolesubmanifold}, and using the left-hand side of \eqref{secondorder} afterwards, we conclude that 
\begin{align*}
        \sys_1(M_\p) &\le \sys(T_\p)\\
        %\sys_{1}(T_{\mathfrak{p}})
        &\leq\frac{4}{(n+1)(2n+3)}\log(\vol(M_{\mathfrak{p}}))+d',
\end{align*}
for some constant $d'$ independent of $\mathfrak{p}$. 
%The result now follows from the fact that $\sys_{1}(M_{\mathfrak{p}})\leq\sys_{1}(T_{\mathfrak{p}}).$
\end{proof}

%1\todo[inline]{Seguir viendo las dos ultimos secciones}

\bibliographystyle{plain}

\end{document}